\documentclass[12pt]{article}

\usepackage{amsmath,amssymb,amsthm}
\usepackage[top=3.2cm,bottom=2.7cm,left=2.5cm,right=2.5cm]{geometry}

\newtheorem{theorem}{Theorem}[section]

\newtheorem{corollary}[theorem]{Corollary}
\newtheorem{lemma}[theorem]{Lemma}

\numberwithin{equation}{section}

\def\C{{\mathbb C}}

\def\N{{\mathbb N}}

\def\E{{\mathbb E}}

\def\P{{\mathbb P}}

\def\eps{\varepsilon}

\begin{document}

\title{Natural boundary and Zero distribution of random polynomials in smooth domains}

\author{Igor Pritsker and Koushik Ramachandran}

\date{}

\maketitle

\begin{abstract}
We consider the zero distribution of random polynomials of the form $P_n(z) = \sum_{k=0}^n a_k B_k(z)$, where $\{a_k\}_{k=0}^{\infty}$ are non-trivial i.i.d. complex random variables with mean $0$ and finite variance. Polynomials $\{B_k\}_{k=0}^{\infty}$ are selected from a standard basis such as Szeg\H{o}, Bergman, or Faber polynomials associated with a Jordan domain $G$ whose boundary is $C^{2, \alpha}$ smooth. We show that the zero counting measures of $P_n$ converge almost surely to the equilibrium measure on the boundary of $G$. We also show that if $\{a_k\}_{k=0}^{\infty}$ are i.i.d. random variables, and the domain $G$ has analytic boundary, then for a random series of the form $f(z) =\sum_{k=0}^{\infty}a_k B_k(z),$ $\partial{G}$ is almost surely a natural boundary for $f(z).$
\end{abstract}

\section{Introduction}
This work is a sequel to \cite{PR} where we showed that zeros of a sequence of random polynomials $\{P_n\}_{n}$ (spanned by an appropriate basis) associated to a Jordan domain $G$ with analytic boundary $L,$ \emph{equidistribute} near $L,$ i.e. distribute according to the equilibrium measure of $L.$ We refer the reader to \cite{PR} for references to the literature on random polynomials. In this note, we extend the above result to Jordan domains with lesser regularity, namely domains with $C^{2, \alpha}$ boundary, see Theorem \ref{P} below.

\vspace{0.1in}

\noindent  In order to state our results we need to set up some notation. Let $G\subset{\C}$ be a Jordan domain. We set $\Omega=\overline{\C} \setminus\overline{G},$ the exterior of $\overline{G}$ and $\Delta$ the exterior of the closed unit disc. By the Riemann mapping theorem there is a unique conformal mapping $\Phi:\Omega\to\Delta,\ \Phi(\infty)=\infty,\ \Phi'(\infty)>0.$ We denote the equilibrium measure of $E = \overline{G}$ by $\mu_E.$ For a polynomial $P_n$ of degree $n,$ with zeros at $\{Z_{k,n}\}_{k=1}^n,$ let $\tau_n = \dfrac{1}{n}\sum_{k =1}^{n}\delta_{Z_{k,n}}$ denote its normalized zero counting measure. For a sequence of positive measures $\{\mu_n\}_{n=1}^{\infty},$ we write $\mu_n \stackrel{w}{\rightarrow}\mu$ to denote weak convergence of these measures to $\mu.$ A random variable $X$ is called non-trivial if $\P(X=0)<1$.

\begin{theorem}\label{P}
Let $G$ be a Jordain domain in $\mathbb{C}$ whose boundary $L$ is $C^{2, \alpha}$ smooth for some $0 < \alpha < 1.$ Consider a sequence of random polynomials $\{P_n\}_{n=0}^{\infty}$ defined by $P_n(z) = \sum_{k=0}^{n}a_kB_k(z),$ where the $\{a_i\}_{i=0}^{\infty}$ are non-trivial i.i.d. random variables with mean $0$ and finite variance, with the basis $\{B_n\}_{n=0}^{\infty}$ being given by either by Szeg\H{o}, or by Bergman, or by Faber polynomials. Then, $\tau_n \stackrel{w}{\rightarrow} \mu_E$ a.s.
\end{theorem}

\noindent We summarize some useful facts obtained in the proof of Theorem \ref{P} below.

\begin{corollary} \label{cor2.2}
Suppose that $E$ is the closure of a Jordan domain $G$ with $C^{2, \alpha}$ boundary $L,$ and that the basis $\{B_k\}_{k=0}^{\infty}$ is given either by Szeg\H{o}, or by Bergman, or by Faber polynomials. If $\{a_k\}_{k=0}^{\infty}$ are non-trivial i.i.d. complex random variables with mean $0$ and finite variance, then the random polynomials $P_n(z) = \sum_{k=0}^n a_k B_k(z)$ converge almost surely to a random analytic function $f$ that is not identically zero. Moreover,
\begin{align} \label{2.1}
\lim_{n\to\infty} |P_n(z)|^{1/n} = |\Phi(z)|, \quad z\in\Omega,
\end{align}
holds with probability one.
\end{corollary}

\noindent As a consequence of Theorem \ref{P}, we show that the zeros of the sequence of derivatives $\{P_n'\}_{n=0}^{\infty}$ also equidistribute.

\begin{corollary}\label{P'}
Let $G,$ $\{a_i\}_{i=0}^{\infty}$ and $P_n$ be as in Theorem \ref{P}. Let $\tau_n'$ denote the zero counting measures of $P_n'.$ Then, $\tau_n' \stackrel{w}{\rightarrow} \mu_E$ a.s.
\end{corollary}

\noindent The natural boundary for a random power series of the form $\sum_{k=0}^{\infty}a_kz^k$ where $\{a_k\}_{k=0}^{\infty}$ are i.i.d random variables has been investigated by quite a few authors. We refer especially to \cite{AR}, but see also \cite{Ka} and the references therein. The result there is that for such a random series, the circle of convergence is a.s. the natural boundary. Some extensions are possible when the $\{a_k\}_{k=0}^{\infty}$ are merely independent. Therefore it seems reasonable to ask if such a result holds when the random series is formed by other polynomial basis. In \cite{PR}, we remarked (without proof) that the random series formed by the basis $\{B_k\}_{k=0}^{\infty},$ has natural boundary $L.$ We prove that result here.   

\begin{theorem}\label{natural}
Suppose that $E$ is the closure of a Jordan domain $G$ with analytic boundary $L,$ and that the basis $\{B_k\}_{k=0}^{\infty}$ is given either by Szeg\H{o}, or by Bergman, or by Faber polynomials. Assume further that the random coefficients $\{a_k\}_{k=0}^{\infty}$ are non-trivial i.i.d. complex random variables satisfying $\E[\log^+|a_0|]<\infty.$ Then the series $$\sum_{k=0}^{\infty}a_k B_k(z)$$ converges a.s. to a random analytic function $f\not\equiv 0$ in $G,$ and moreover, with probability one, $\partial G = L$ is the natural boundary for $f.$
\end{theorem}

\section{Proofs}

\begin{proof}[Proof of Theorem \ref{P} and of Corollary \ref{cor2.2}]
We closely follow the ideas in \cite{PR}. The proof consists of two probabilistic lemmas followed by the use of a deterministic theorem in potential theory. The first lemma below follows from a standard application of the Borel-Cantelli lemma.

\begin{lemma}\label{lem3.1}
If $\{a_k\}_{k=0}^{\infty}$ are non-trivial, independent and identically distributed complex random variables that satisfy $\E[\log^+|a_0|]<\infty$, then
\begin{align} \label{3.1}
\limsup_{n\to\infty} |a_n|^{1/n}= 1 \quad\mbox{ a.s.},
\end{align}
and
\begin{align} \label{3.2}
\limsup_{n\to\infty} \left(\max_{0\le k\le n} |a_k| \right)^{1/n} = 1 \quad\mbox{ a.s.}
\end{align}
\end{lemma}

\noindent A slightly more delicate application of Borel-Cantelli gives the following result. For the proof, we refer to \cite{PR}.
\begin{lemma}\label{lem3.2}
If $\{a_k\}_{k=0}^{\infty}$ are non-trivial i.i.d. complex random variables, then
there is a $b>0$ such that
\begin{align} \label{3.3}
\liminf_{n\to\infty} \left(\max_{n-b\log{n}<k\le n} |a_k|\right)^{1/n} \ge 1 \quad\mbox{ a.s.}
\end{align}
\end{lemma}

\noindent We use the following theorem of Grothmann \cite{Gr} which describes the zero distribution of deterministic polynomials. 

\vspace{0.1in}

\noindent Let $E\subset\C$ be a compact set of positive capacity such that $\Omega=\overline{\C} \setminus E$ is connected and regular. The Green function of $\Omega$ with pole at $\infty$ is denoted by $g_\Omega(z,\infty)$. We use $\|\cdot\|_K$ for the supremum norm on a compact set $K$.

\noindent
{\bf Theorem G.} {\em If a sequence of polynomials $P_n(z),\ \deg(P_n)\le n\in\N,$ satisfies
\begin{align} \label{3.4}
\limsup_{n\to\infty} \|P_n\|_E^{1/n} \le 1,
\end{align}
for any closed set $K\subset E^\circ$
\begin{align} \label{3.5}
\lim_{n\to\infty} \tau_n(K) = 0,
\end{align}
and there is a compact set $S\subset\Omega$ such that
\begin{align} \label{3.6}
\liminf_{n\to\infty} \max_{z\in S} \left(\frac{1}{n} \log|P_n(z)| - g_\Omega(z,\infty) \right) \ge 0,
\end{align}
then the zero counting measures $\tau_n$ of $P_n$ converge weakly to $\mu_E$ as $n\to\infty.$}

\vspace{0.1in}

\noindent The idea now is to check that with probability $1,$ our sequence of polynomials satisfies the hypothesis in Grothmann's theorem. 

\vspace{0.1in}

\noindent Note that \eqref{3.4} is satisfied for $E$ almost surely by \eqref{3.2}, and the estimate
\[
\|P_n\|_E \le \sum_{k=0}^n |a_k| \|B_k\|_E \le (n+1) \max_{0\le k\le n} |a_k| \, \max_{0\le k\le n} \|B_k\|_E,
\]
as
\[
\limsup_{n\to\infty} \left(\max_{0\le k\le n} \|B_k\|_E\right)^{1/n} \le 1.
\]

This last fact follows from the well known result that in all three cases of polynomial bases we consider in this theorem, we have
\begin{align} \label{3.7}
\lim_{n\to\infty} |B_n(z)|^{1/n} = |\Phi(z)|
\end{align}
holds uniformly on compact subsets of $\Omega.$ To check that \eqref{3.5} holds, we use the following lemma from \cite{HKPV}

\begin{lemma}\label{lem 3.4}
Let $\psi_n$ be holomorphic functions on a domain $\Lambda.$ Assume that $\sum_{n=0}^{\infty}|\psi_n|^2$ converges uniformly on compact sets of $\Lambda.$ Let $a_n$ be i.i.d. random variables with zero mean and finite variance. Then, almost surely, $\sum_{n=0}^{\infty}a_n\psi_n(z)$ converges uniformly on compact subsets of $\Lambda$ and hence defines a random analytic function.
\end{lemma}

\noindent It is well known that $\sum_{n=0}^{\infty}|B_n(z)|^2 = K(z,z),$ where $K(z,w)$ denotes the Bergman (or correspondingly Szeg\H{o}) kernel of the domain $G$ when $\{B_i\}_{i=0}^{\infty}$ denotes the Bergman or Szeg\H{o} basis respectively. For the case of the Faber polynomials, the convergence follows from the estimates of the $\sup$ norm $||P_n||_K$ on any compact $K\subset G$, see \cite{Su}. With this knowledge, taking $\psi_n = B_n$ and $\Lambda= G$ in Lemma \ref{lem 3.4}, we obtain that almost surely, $\sum_{n=0}^{\infty}a_nB_n(z)$ converges uniformly on compact subsets of $G$ and hence defines a random analytic function $f.$ The uniqueness of series expansions of these polynomial basis ensures that $f$ is not identically $0.$ Since $P_n\to f,$  an application of Hurwitz's theorem from basic complex analysis now proves \eqref{3.5}. Incidentally this also proves the corresponding part of Corollary \ref{cor2.2}

If $\tau_n$ do not converge to $\mu_E$ a.s., then \eqref{3.6} cannot hold a.s. for any compact set  $S$ in $\Omega.$ We choose $S=L_R =\{z: g(z)=R\},$ with $R>1$, and find a subsequence $n_m,\ m\in\N,$ such that
\begin{align} \label{3.8}
\limsup_{m\to\infty} \|P_{n_m}\|_{L_R}^{1/n_m} < R,
\end{align}
holds with positive probability. It follows from a result of Suetin \cite{Su} that for Bergman polynomials,
\begin{equation}\label{3.11}
B_n(z) = \sqrt{\frac{n+1}{\pi}}\Phi^n(z)\Phi'(z)\left(1 + A_n(z)\right),
\end{equation}
holds locally uniformly in $\Omega$  where we recall that $\Phi$ is the exterior conformal map, $\Phi:\Omega\to\Delta,\ \Phi(\infty)=\infty,\ \Phi'(\infty)>0,$ and 

\begin{equation}\label{error}
|A_n(z)|\leq c \dfrac{\log(n)}{n^2}.
\end{equation} 

Similar asymptotic formulas as \eqref{3.11} are valid for Szeg\H{o} and Faber polynomials but without the factor $\sqrt{n+1}.$ The proofs for these basis have to  accordingly modified. Equation \eqref{3.11} implies that all zeros of $B_n$ are contained inside $L_R$ for all large $n$. This allows us to write an integral representation
\begin{align} \label{3.9}
a_n = \frac{1}{2\pi i} \int_{L_R} \frac{P_n(z)\,dz}{z B_n(z)},
\end{align}
which is valid for all large $n\in\N$ because $P_n(z)/(z B_n(z)) = a_n/z + O(1/z^2)$ for $z\to\infty.$ The asymptotic on  $B_n$ from \eqref{3.11} implies that there are positive constants $c_1$ and $c_2$ that do not depend on $n$ and $z$, such that
\begin{align} \label{3.10}
c_2 \sqrt{n}\, \rho^n \le |B_n(z)| \le c_1 \sqrt{n}\, \rho^n, \quad z\in L_{\rho},\ \rho>1,\ n\in\N.
\end{align}
We estimate from \eqref{3.9} and \eqref{3.10} with $\rho=R$ that
\[
|a_n| \le \frac{|L_R|}{2\pi d} \frac{\|P_n\|_{L_R}}{c_2 \sqrt{n}\, R^n},
\]
where $|L_R|$ is the length of $L_R$ and $d:=\min_{z\in L_R} |z|.$ It follows that
\[
\|P_{n-1}\|_{L_R} \le \|P_n\|_{L_R} + |a_n| \|B_n\|_{L_R} \le \|P_n\|_{L_R} \left(1 + \frac{|L_R|}{2\pi d} \frac{c_1}{c_2} \right) =: C\, \|P_n\|_{L_R}, \quad n\in\N.
\]
Applying this estimate repeatedly, we obtain that
\[
\|P_{n-k}\|_{L_R} \le  C^k\, \|P_n\|_{L_R}, \quad k\le n,
\]
so that \eqref{3.9} yields
\[
|a_{n-k}| \le \frac{|L_R|}{2\pi d} \frac{\|P_{n-k}\|_{L_R}}{c_2 \sqrt{n-k}\, R^{n-k}} \le \frac{|L_R|}{2\pi d} \frac{C^k\, \|P_n\|_{L_R}}{c_2 \sqrt{n-k}\, R^{n-k}}.
\]
Choosing sufficiently small $\eps>0$ and using \eqref{3.8}, we deduce from previous inequality that
\[
|a_{n_m-k}| \le q^{n_m}, \quad 0\le k\le \eps n_m,
\]
for some $q\in(0,1)$ and all sufficiently large $n_m$, with positive probability. The latter estimate clearly contradicts \eqref{3.3} of Lemma \ref{lem3.2}. Hence \eqref{3.6} holds for $S=L_R,$ with any $R>1$, and $\tau_n$ converge weakly to $\mu_E$ with probability one. Note that \eqref{3.6} for $S=L_R,$ with $R>1$, is equivalent to \eqref{2.1}. Indeed, we have equality in \eqref{3.6}, with $\lim$ instead of $\liminf$, by Bernstein-Walsh inequality and \eqref{3.4}, see Remark 1.2 of \cite[p. 51]{AB} for more details. This concludes the proof of Theorem \ref{P} as well as the proof of Corollary \ref{cor2.2}.

\end{proof}

\begin{proof}[Proof of Corollary \ref{P'}]

The method of proof is similar to that of Theorem \ref{P}, namely check that the conditions in Grothmann's result hold almost surely. First, we use a Markov-Bernstein result (cf. \cite{IP} and the references therein) to bound the sup norm of $P_n'$ on $E.$ 

\begin{equation}\label{MarkovB}
||P_n'||_{E}\leq c(E)n^2||P_n||_E.
\end{equation}

\noindent Therefore with probability one, $$\limsup_{n\to\infty} \|P_n'\|_E^{1/(n-1)}\leq \limsup_{n\to\infty} \left(c(E)n^2\|P_n\|_E\right)^{1/(n-1)} \le 1.$$  

\noindent This shows that \eqref{3.4} holds for $P_n'.$ Next, we know from the proof of Theorem \ref{P} that with probability one, $P_n\to f$ uniformly on compacts, where $f$ is a nonzero random analytic function. From this we obtain that $P_n'\to f'$ also uniformly on compacts. The function $f'$ is not identically $0,$ for if it were, $f\equiv c$ for some constant $c,$ and by the uniqueness of series expansion for the polynomial basis under consideration, this would imply that $a_i =0$ for $i\geq 1.$  This contradicts Lemma \ref{lem3.2}. From here, an application of Hurwitz's theorem now yields that $\tau_n'(K)\to 0$ for every compact set $K\subset G.$ This proves equation \eqref{3.5} for $P_n'.$ Finally, recall that $$B_n(z) = \sqrt{\frac{n+1}{\pi}}\Phi^n(z)\Phi'(z)\left(1 + A_n(z)\right)$$ 
where $A_n$ satisfies the estimate \eqref{error}. Differentiating this, we obtain bounds for $B_n'$ on $L_R.$ Namely

\begin{equation}\label{B_n'}
c_4n^{\frac{3}{2}}R^{n-1}\leq |B_n'(z)|\leq c_5n^{\frac{3}{2}}R^{n-1}.
\end{equation}
 
\noindent To obtain this asymptotic, we have used a local Cauchy integral to estimate $A_n'.$ 
$$A_n'(z) = \frac{1}{2\pi i}\int_{\partial B_{\delta}(z)}\dfrac{A_n(w)}{(z-w)^2}dw$$

\noindent for $z\in L_R$ with $\delta > 0$ being chosen so that the ball $B_{\delta}(z)$ stays away from the boundary, say $\delta = \frac{1}{5}d(L_R, L).$ Using the uniform bound \eqref{error} in the above integral shows that an analogous estimate holds for $A_n'.$ Once we obtain \eqref{B_n'}, we note that the proof for \eqref{3.6} for $P_n'$ follows as in Theorem \ref{P}. All the conditions in Grothmann's theorem are satisfied and hence we have the required convergence. 
\end{proof}

\noindent \textbf{Remark:} Although Theorem \ref{P} and Corollary \ref{P'} have been stated for Jordan domains with $C^{2, \alpha}$ boundary, it is easy to see that the same proof goes through if for instance $G$ is a Jordan domain whose boundary is piecewise analytic (with angles at the corners satisfying certain conditions). The asymptotic equation \eqref{3.11} will then have to be replaced by an analogous one for piecewise analytic boundary, see \cite{ST} and the references therein.

\begin{proof}[Proof of Theorem \ref{natural}]
We have that $E$ is the closure of a Jordan domain $G$ bounded by an analytic curve $L$ with exterior $\Omega.$ It is well known that the conformal mapping $\Phi:\Omega\to\Delta,\ \Phi(\infty)=\infty,\ \Phi'(\infty)>0,$ extends through $L$ into $G$, so that $\Phi$ maps a domain $\Omega_r$ containing $\overline{\Omega}$ conformally onto $\{|z|>r\}$ for some $r\in(0,1).$ In particular, the level curves of $\Phi$ denoted by $L_\rho$ are contained in $G$ for all $\rho\in(r,1)$, $L_1=L$ and $L_\rho\subset\Omega$ for $\rho>1.$ 

\vspace{0.1in}

\noindent For the proof that the series $\sum_{k=0}^{\infty}a_k B_k(z)$ converges a.s. to an analytic function $f,$ we refer the reader to Corollary $2.2$ of \cite{PR}.

\vspace{0.1in}

\noindent We now show the result about $L$ being the natural boundary of $f.$ We will give the proof for the basis of Faber and Bergman polynomials. The proof for the Szeg\H{o} polynomials is similar to the Bergman case but simpler. 

\vspace{0.1in}

\noindent Let
 $$\Phi(z) = \frac{z}{\text{cap}(E)} + \sum_{k=1}^{\infty}\frac{c_k}{z^k},$$ 
for $z$ in a neighborhood of infinity. Let $F_n$ be the $n$th Faber polynomial. By definition, $F_n$ is the polynomial part of the Laurent expansion of $\Phi^n$ at infinity, 

\begin{equation}\label{F_n}
\Phi^n(z) = F_n(z) + E_n(z),\hspace{0.1in} z\in\Omega_r,
\end{equation}
\noindent where $E_n$ is analytic, consisting of all the negative powers of $z$ in the expansion of $\Phi^n.$ Fix $\epsilon >0$ such that $r+\epsilon < 1.$ It follows that

$$E_n(z) = \frac{1}{2\pi i}\int_{\Gamma_{r+\epsilon}}\frac{\Phi^n(t)}{t - z}dt, \hspace{0.05in} z\in\Omega_{\rho},$$

\noindent for $r+\epsilon < \rho.$ From the above integral representation it is clear that 

\begin{equation}\label{estimateE}
|E_n(z)|\leq\dfrac{|\Gamma_{r+\epsilon}|(r+\epsilon)^n}{2\pi d(\Gamma_{r+\epsilon}, \Gamma_{\rho})}
\end{equation}

\noindent for $z\in\Omega_{\rho}.$ Here $d(\Gamma_{r+\epsilon}, \Gamma_{\rho})$ denotes the distance between $\Gamma_{r+\epsilon}$ and $\Gamma_{\rho}.$ Using \eqref{estimateE} and the fact that $\limsup |a_n|^{\frac{1}{n}} = 1$ a.s. (see equation \eqref{3.1} below), we deduce that the series $\sum_{k=0}^{\infty} a_kE_k(z)$ converges a.s. in $\Omega_{\rho}$ and defines a random analytic function there. From equation \eqref{F_n} we know that for $z\in G\cap\Omega_{\rho},$ $r+\epsilon < \rho < 1,$ 
$$\sum_{k=0}^{\infty}a_k F_k(z) = \sum_{k=0}^{\infty}a_k \Phi^k(z) - \sum_{k=1}^{\infty} a_kE_k(z).$$ 
Now suppose that the series $f = \sum_{k=0}^{\infty}a_k F_k(z)$ has an analytic continuation across $L = L_1.$ Then, together with the fact that the second series on the right defines an analytic function in $\Omega_{\rho},$ this implies that $\sum_{k=0}^{\infty}a_k w^k$ has an analytic continuation across $|w| = 1,$ where $w = \Phi(z).$ But this contradicts Satz $8$ of \cite{AR}.

\vspace{0.1in}

\noindent If $\{B_k\}_{k=0}^{\infty}$ denotes the Bergman basis, then Carleman's asymptotic formula (see \cite{Ga}, Chapter $1$), yields

\begin{equation}\label{B}
B_n(z) = \sqrt{\frac{n+1}{\pi}}\Phi^n(z)\Phi'(z)\left(1 + e_n(z)\right)
\end{equation}

\noindent where 
\begin{equation}\label{error1}
e_n(z)=\left\{\begin{array}{ll} O(\sqrt{n})r^n, & z\in L_{\rho}, \hspace{0.05in} \rho > 1\\
O(\frac{1}{\sqrt{n}})(\frac{r}{\rho})^n & z\in L_{\rho},\hspace{0.05in} r <\rho < 1.\end{array}\right.
\end{equation}

Using $\limsup |a_n|^{\frac{1}{n}} = 1$ a.s. and estimates \eqref{error1}, it is not hard to see that the series $\sum_{n=0}^{\infty}a_n\sqrt{\frac{n+1}{\pi}}\Phi^n(z)\Phi'(z)e_n(z)$ converges a.s. in a neighborhood of the boundary $L,$ and defines an analytic function there. Now from \eqref{B}, we have

$$\sum_{n=0}^{\infty}a_nB_n(z) = \sum_{n=0}^{\infty}a_n\sqrt{\frac{n+1}{\pi}}\Phi^n(z)\Phi'(z) + \sum_{n=0}^{\infty}a_n\sqrt{\frac{n+1}{\pi}}\Phi^n(z)\Phi'(z)e_n(z)$$

\noindent for $z\in G\cap\Omega_{\rho},$ $r < \rho < 1.$  If the series $\sum_{n=0}^{\infty}a_nB_n(z)$ has an analytic continuation across $L,$ then combined with the fact that the second series on the right defines an analytic function near $L,$ we would obtain that $\sum_{n=0}^{\infty}a_n\sqrt{\frac{n+1}{\pi}}\Phi^n(z)\Phi'(z)$ and hence $\sum_{n=0}^{\infty}a_n\sqrt{\frac{n+1}{\pi}}\Phi^n(z)$ has an analytic continuation across $L.$ In other words, taking $w = \Phi(z)$ the series $\sum_{n=0}^{\infty}a_n\sqrt{\frac{n+1}{\pi}}w^n$ has an analytic continuation across $|w| = 1.$ This contradicts Satz $12$ in \cite{AR}, and finishes the proof. 

\end{proof}

\section*{Acknowledgments} Research of the first author was partially supported by the National Security Agency (grant H98230-15-1-0229) and by the American Institute of Mathematics.

\bigskip
\noindent Department of Mathematics, Oklahoma State University, Stilwater, OK 74078, USA\\
\textit{igor@math.okstate.edu} (Pritsker),\\
\textit{koushik.ramachandran@okstate.edu} (Ramachandran)
\end{document}